\documentclass{amsart}
\usepackage{amssymb,amsmath,amsthm}
\usepackage{mathrsfs}
\usepackage{amsfonts}

\newtheorem{theorem}{Theorem}[section]
\newtheorem{lemma}[theorem]{Lemma}
\theoremstyle{definition}
\newtheorem{definition}[theorem]{Definition}

\newtheorem{corollary}[theorem]{Corollary}

\theoremstyle{remark}
\newtheorem{remark}[theorem]{Remark}

\numberwithin{equation}{section}

\begin{document}

\title{A remark on $\mathbb{T}$-valued cohomology groups of algebraic group actions}

\author{Yongle Jiang}
\address{Yongle Jiang, Department of Mathematics, SUNY at Buffalo, NY \text{14260-2900}, U.S.A.}
\email{yongleji@buffalo.edu}



%


\begin{abstract}

We prove that for a weakly mixing algebraic action $\sigma: G\curvearrowright(X,\nu)$, the $n$th cohomology group $H^n(G\curvearrowright X; \mathbb{T})$, after quotienting out the natural subgroup $H^n(G,\mathbb{T})$, contains $H^n(G,\widehat{X})$ as a natural subgroup for $n=1$. If we further assume the diagonal actions $\sigma^2, \sigma^4$ are $\mathbb{T}$-cocycle superrigid and $H^2(G, \widehat{X})$ is torsion free as an abelian group, then the above also holds true for $n=2$. Applying it for principal algebraic actions when $n=1$, we show that $H^2(G,\mathbb{Z}G)$ is torsion free as an abelian group when $G$ has property (T) as a direct corollary of Sorin Popa's cocycle superrigidity theorem; we also use it (when $n=2$) to answer, negatively, a question of Sorin Popa on the 2nd cohomology group of Bernoulli shift actions of property (T) groups.

\end{abstract}

\maketitle

\section{Introduction}

Let $G$ be a countable infinite discrete group with integer group ring $\mathbb{Z}G$. Recall that an \textbf{algebraic $G$-action} is a homomorphism $\alpha: G\to Aut(X)$ from $G$ to the group of (continuous) automorphisms of a compact metrizable abelian group $X$. By duality, such an action of $G$ is determined by a $\mathbb{Z}G$-module $\widehat{X}$ (see \cite[Section 1]{Li-Peterson-Schmidt} for details). When $\widehat{X}=\frac{\mathbb{Z}G}{\mathbb{Z}Gf}$ for some $f\in\mathbb{Z}G$, the corresponding algebraic action is called \textbf{a principal algebraic $G$-action}. A central question on a principal algebraic action is whether it is measurable conjugate to a Bernoulli shift under natural assumptions. It is trivial to observe that the answer is yes when assuming $f=k\in \mathbb{Z}G$ for some integer $k$. However, it is a highly non-trivial theorem saying that the answer is yes when $G=\mathbb{Z}^d$ for all positive integer $d$ if we assume the action has completely positive entropy \cite{Schmidt-Rudolph}. When $G$ is non-abelian, this question is still wide open.

Towards this problem, it is natural to first investigate whether Bernoulli shifts and principal algebraic actions share any common dynamical properties, such as ergodicity, weakly mixing, mixing, freeness, completely positive entropy etc. However, it turns out that even these weaker questions are usually difficult to settle. We refer the reader to \cite{Bowen_Li, Ben, Kerr} for research on these questions. Very recently, cohomological methods have been employed to study ergodicity of principal algebraic actions in \cite{Li-Peterson-Schmidt}.

However, the entropy theory of principal algebraic actions has received much attention in the past decade from amenable groups case in \cite{Den, DS, Li, Li-Thom, Schmidt} to sofic groups in \cite{Bow4, Bowen_Li, B.Hayes} due to the emergence of sofic entropy in \cite{Bowen_sofic, Kerr_Li}.

Our motivation behind this paper is to study the above questions from the viewpoint of von Neumann algebras, following \cite{M_R_Vaes, Popa_Vaes_1, Popa_Vaes_2} in spirit.

In the theory of von Neumann algebras, especially II$_1$ factors, a fundamental theme is to study various rigidity phenomena such as finding $\mathbb{T}$-cocycle superrigid actions, where $\mathbb{T}$ denotes the unit circle. Recall that a free, ergodic, probability measure preserving (p.m.p.) action of a countable discrete group $G$ on a standard probability space $(X,\nu)$ is said to be \textbf{$\mathbb{T}$-cocycle superrigid} if any $\mathbb{T}$-valued 1-cocycle for the action $G\curvearrowright (X,\nu)$ must be cohomologous (see Definition \ref{cohomologous}) to a homomorphism from $G$ to $\mathbb{T}$.

Denote by \textbf{$\mathcal{G}$} the class of groups $G$ such that the standard Bernoulli shift action $G\curvearrowright\prod_{g\in G}(X_0, \mu_0)$ is $\mathbb{T}$-cocycle superrigid, where $(X_0,\mu_0)$ is any standard probability space. The following (non-amenable) groups are shown to be in $\mathcal{G}$ mainly by Sorin Popa's powerful deformation/rigidity theory \cite{Popa_4} and Jesse Peterson's $L^2$-rigidity theory \cite{Peterson}.

\begin{itemize}
  \item \cite{Sysak-Popa} $G$ is weakly rigid, i.e.  there exists an infinite normal subgroup $H<G$ such that $(G, H)$ has relative property (T); in particular, when $G$ has property (T).

  \item \cite{Peterson-Sinclair} When $L(G)$ is a $L^2$-rigid II$_1$ factor in the sense of Peterson \cite{Peterson}; for example, when $L(G)$ is a II$_1$ factor which is either non-prime or has property $(\Gamma)$.

  \item \cite{Popa_3} If $G$ has an infinite rigid subgroup $H$ such that $H$ is wq-normal or w-normal.

  \item \cite{Popa_2} $G=F_2\times F_2$, where $F_2$ is the free group on two letters. For the general version, see \cite[Theorem 1.1, Corollary 1.2]{Popa_2}.

  \item \cite{Robin Tucker-Drob} $G$ is inner amenable but not amenable.
\end{itemize}

Therefore, one natural question to ask is whether a principal algebraic $G$-action (or in general an algebraic action) is $\mathbb{T}$-cocycle superrigid when $G\in\mathcal{G}$ under natural assumptions.

Motivated by this question and inspired by \cite[Corollary 4.2]{Peterson-Sinclair}\cite[Proof of Proposition 5.2]{K.Schmidt_2}, we prove the following theorem.
\begin{theorem}\label{maintheorem}
Let $G$ be a countable infinite discrete group, $X$ be a compact metrizable abelian group with Haar measure $\nu$ and $\sigma: G\curvearrowright (X, \nu)$ be an algebraic action. If $\sigma$ is weakly mixing as a p.m.p. action, then we have an injective group homomorphism
$$H^1(G, \widehat{X})\hookrightarrow \frac{H^1(G\curvearrowright X;\mathbb{T})}{H^1(G, \mathbb{T})}.$$
Under further assumptions that the diagonal actions $\sigma^2:G\curvearrowright(X^2, \nu^2)$, $\sigma^4:G\curvearrowright(X^4, \nu^4)$ are $\mathbb{T}$-cocycle superrigid and $H^2(G, \widehat{X})$ is torsion free as an abelian group, we also have an injective group homomorphism
$$H^2(G, \widehat{X})\hookrightarrow \frac{H^2(G\curvearrowright X;\mathbb{T})}{H^2(G, \mathbb{T})}.$$
\end{theorem}
\begin{remark}
(a) Under the assumptions in the theorem, $H^n(G, \mathbb{T})$ is identified with a subgroup of $H^n(G\curvearrowright X;\mathbb{T})$ for $n=1, 2$ by the natural map described in Lemma \ref{lemma2} below.

(b) In the proof of the second embedding, we only need to assume $H^2(G, \widehat{X})$ has no non-trivial elements of order two (rather that being torsion free). Without the assumption that $H^2(G, \widehat{X})$ is torsion free as an abelian group, a small modification of our proof actually shows that the second embedding above still holds if we quotient out corresponding torsion subgroups from both sides, which was pointed out to us by Hanfeng Li.
\end{remark}

The paper is organized as follows. In Section \ref{section2}, we first recall the definitions of weakly mixing actions, $n$-cocycles and $n$th cohomology groups in the setting of group actions and modules, then we prove two lemmas; one is used to replace almost group homomorphisms with genuine ones; the other one justifies $H^n(G, \mathbb{T})$ being a natural subgroup of $H^n(G\curvearrowright X;\mathbb{T})$ for $n=1, 2$ under assumptions. In Section \ref{section3}, we prove Theorem \ref{maintheorem}. In Section \ref{section4}, we give applications of Theorem \ref{maintheorem} as mentioned in the abstract. We conclude with further questions related to this paper in Section \ref{section5}.
\\
\\
\noindent{\it Acknowledgments.}
We are indebted to Hanfeng Li for constant support and pointing out errors in a previous version of this paper. We thank Christian Rosendal for sharing his unpublished notes on almost group homomorphisms at an early stage of this paper. We also thank Jesse Peterson for his interest in this work, and thank Sorin Popa for helpful comments. Also thanks to Yves de Cornulier, Mike Mihalik, Boris Okun and Ben Wieland for answering our questions on $H^2(G, \mathbb{Z}G)$. Last but not least, we are grateful to the referee for extremely helpful suggestions and comments which improves the paper greatly.

\section{Preliminaries}\label{section2}
First, let us recall the well-known definition of weakly mixing actions.

\begin{definition}
A probability measure preserving action $G\curvearrowright (X, \mu)$ is said to be \textbf{weakly mixing} if one of the following equivalent conditions is satisfied.

\begin{itemize}
  \item The diagonal action $G\curvearrowright (X\times X, \mu\times \mu)$ is ergodic.
  \item The diagonal action $G\curvearrowright (X\times Y, \mu\times \eta)$ is ergodic for all ergodic $G\curvearrowright (Y, \eta)$.
  \item The constant functions form the only finite-dimensional $G$-invariant subspace of $L^2(X, \mu)$.
\end{itemize} 
Note that from the above definition, it follows that if $G\curvearrowright (X, \mu)$ is weakly mixing, then the diagonal $G$-action on the product $(X^n, \mu^n)$ is ergodic, and weakly mixing for all $n\geq 2$.

\end{definition} 

Next, let us recall the definition of cocycles in the group action setting \cite{FM}.

\begin{definition}
Let $G\curvearrowright(X,\mu)$ be a p.m.p. group action, we say a measurable map $c: G^n\times X\to\mathbb{T}$ is a $n$-cocycle if for all $g_1, g_2, \cdots,g_{n+1}\in G$, the equation
\begin{eqnarray*}
c(g_2, g_3, \dots, g_{n+1}, g_1^{-1}x)c(g_1g_2, g_3,\dots, g_{n+1}, x)^{-1}c(g_1, g_2g_3, g_4, \dots, g_{n+1}, x)\cdots\\
\cdots c(g_1, g_2, \dots, g_{n-1}, g_ng_{n+1}, x)^{(-1)^n}c(g_1, g_2, \dots, g_n, x)^{(-1)^{n+1}}=1
\end{eqnarray*}
holds true $\mu$-almost every $x\in X$.
\end{definition}

Since $G$ is countable, in the above definition, we can further assume there exists a conull subset $Z$ of $X$ such that the above equation holds true for all $ x\in Z$ and all $g_1, g_2, \cdots,g_{n+1}\in G$. Here, $Z\subset (X, \mu)$ is conull if $\mu(X\setminus Z)=0$.

\begin{definition}\label{cohomologous}
Two $n$-cocycles $c_1, c_2: G^n\times X\to\mathbb{T}$ are called cohomologous if there is a measurable map $b: G^{n-1}\times X\to\mathbb{T}$ such that for all $g_1, g_2, \cdots,g_{n}\in G$,
\begin{eqnarray*}
c_1(g_1,\dots, g_n, x)b(g_2,\dots, g_{n}, g_1^{-1}x)b(g_1g_2, g_3,\dots, g_{n}, x)^{-1}b(g_1, g_2g_3,\dots, g_{n}, x)\cdots\\
\cdots b(g_1,\dots, g_{n-2}, g_{n-1}g_{n}, x)^{(-1)^{n-1}}b(g_1, g_2,\dots, g_{n-1}, x)^{(-1)^{n}}=c_2(g_1,\dots, g_n, x)
\end{eqnarray*}
holds true $\mu$-almost every $x\in X$.
\end{definition}

\begin{definition}
We say a $n$-cocycle is trivial if it is cohomologous to the cocycle which takes the constant value $1\in\mathbb{T}$.
\end{definition}

The set of all $n$-cocycles for the action $G\curvearrowright(X,\mu)$ with values in $\mathbb{T}$ is denoted by $Z^n(G\curvearrowright X; \mathbb{T})$, the set of trivial cocycles is denoted by $B^n(G\curvearrowright X; \mathbb{T})$, and the set of equivalence class of cohomologous cocycles is denoted by $H^n(G\curvearrowright X; \mathbb{T})$. Note that since $\mathbb{T}$ is abelian, $Z^n(G\curvearrowright X; \mathbb{T})$ is an abelian group under pointwise multiplication, $B^n(G\curvearrowright X; \mathbb{T})$ is a subgroup and $H^n(G\curvearrowright X; \mathbb{T})$ is the quotient group, which is called the $n$th cohomology group for the action $G\curvearrowright(X,\mu)$.

Let $G$ be a group and $M$ be a left $G$-module, the cohomology group $H^n(G, M)$ for any natural number $n$ can also be defined using cocycles \cite[Page 59]{GTM_87}.

\begin{definition}
A map $c: G^n\to M$ is called a $n$-cocycle if the equation
\begin{eqnarray*}
g_1c(g_2, g_3,\dots, g_{n+1})-c(g_1g_2, g_3,\dots, g_{n+1})+c(g_1, g_2g_3, g_4,\dots, g_{n+1})-\cdots\\
\cdots +(-1)^nc(g_1, g_2, \dots, g_{n-1}, g_ng_{n+1})+(-1)^{n+1}c(g_1, g_2,\dots, g_n)=0
\end{eqnarray*}
holds true for all $g_1,\cdots, g_{n+1}$ in $G$.
\end{definition}

\begin{definition}
Two $n$-cocycles $c_1, c_2: G^n\to M$ are called cohomologous if there is a map $b: G^{n-1}\to M$ such that
\begin{eqnarray*}
c_1(g_1,\cdots, g_n)+g_1b(g_2,\cdots, g_{n})-b(g_1g_2, g_3,\dots, g_n)+b(g_1, g_2g_3,\dots, g_n)-\cdots\\
\cdots+(-1)^{n-1}b(g_1,\dots, g_{n-2}, g_{n-1}g_n)+(-1)^nb(g_1, g_2,\dots, g_{n-1})=c_2(g_1,\dots, g_n)
\end{eqnarray*}
holds true for all $g_1, g_2, \cdots,g_{n}\in G$.
\end{definition}

\begin{definition}
We say a $n$-cocycle is trivial if it is cohomologous to the cocycle which takes the constant value $0\in M$.
\end{definition}

In the same way, we define the $n$th cohomology group $H^n(G, M)$ for a $G$-module $M$ as the quotient group $Z^n(G, M)/B^n(G,M)$.
Especially, when $M=\mathbb{Z}G$ with left $G$-multiplication and $M=\mathbb{T}$ with trivial $G$-action, we can define $H^n(G, \mathbb{Z}G)$ and $H^n(G, \mathbb{T})$ respectively.

A general version of the following lemma was communicated to us by Christian Rosendal (see also \cite[Theorem 3]{Moore_1}). We are grateful to Hanfeng Li for the following quick proof.

\begin{lemma}\label{lemma1}
Let $X$ be a compact metrizable abelian group with Haar measure $\nu$, and let $b: X\to\mathbb{T}$ be a measurable map. If $b(x+y)=b(x)b(y)$ holds true for almost every $(x,y)\in X\times X$, then there is a continuous group homomorphism $\lambda: X\to \mathbb{T}$ such that $\lambda(x)=b(x)$ for almost every $x\in X.$
\end{lemma}
\begin{proof}
Note that the translation action $\alpha: X\curvearrowright L^2(X)$ is a continuous group action, and $b\in L^2(X)$ with $||b||_2=1$.
By assumption, we can find a conull subset $Z$ of $X$ such that $\alpha_z(b)=b(z)b$ holds true for all $z\in Z$.
This implies $b$ is an eigenvector for $\alpha_x$ for all $x\in X$, i.e.   $\alpha_x(b)=\lambda(x)b$ for all $x\in X$ and some map $\lambda: X\to\mathbb{T}$. In particular, $\lambda(z)=b(z)$ for all $z\in Z$. It is easy to check that $\lambda: X\to\mathbb{T}$ satisfies the required properties.
\end{proof}
\begin{lemma}\label{lemma2}
For a p.m.p. action $G\curvearrowright (X,\nu)$, we have a natural group homomorphism $\Psi_n: H^n(G, \mathbb{T})\rightarrow H^n(G\curvearrowright X; \mathbb{T})$ for any natural number $n$. It is injective for $n=1$ if the action is weakly mixing; it is also injective for $n=2$ if the diagonal action $\sigma^2: G\curvearrowright (X^2, \nu^2)$ is weakly mixing and $\mathbb{T}$-cocycle superrigid.

\end{lemma}
\begin{proof}
We define $\Psi_n: H^n(G, \mathbb{T})\to H^n(G\curvearrowright X;\mathbb{T})$ by $\Psi_n([c])=[c']$, where $c': G^n\times X\to\mathbb{T}$ is defined by $c'(g_1,\cdots, g_n, x)=c(g_1,\cdots, g_n).$

It is easy to check that $\Psi_n$ is a well-defined group homomorphism.

For $n=1$, we claim $\Psi_1$ is injective if the action is weakly mixing.

Note that $H^1(G, \mathbb{T})=Hom(G, \mathbb{T})$.

If $[c]\in H^1(G, \mathbb{T})$ and $\Psi_1([c])=[0]$, then there exists a measurable map $b: X\to \mathbb{T}$ such that $c(g)=b(g^{-1}x)b(x)^{-1}$ holds a.e. $x\in X$. Now, viewing $b\in L^2(X)$, we have $gb=c(g)b$, so $\mathbb{C}b$ is an invariant subspace under the $G$-action. Now the action is weakly mixing implies $b(x)=\lambda\in\mathbb{T}$ a.e. $x\in X$, so $c(g)=1$ for all $g\in G$, and $[c]=[0]$.

For $n=2$, we claim $\Psi_2$ is injective if the diagonal action $\sigma^2: G\curvearrowright (X^2, \nu^2)$ is weakly mixing and $\mathbb{T}$-cocycle superrigid.

If $[c]\in H^2(G, \mathbb{T})$ and $\Psi_2([c])=[0]\in H^2(G\curvearrowright X;\mathbb{T})$, there exists a measurable map $b: G\times X\to \mathbb{T}$ such that
\begin{eqnarray}\label{2*}
c(g,h)=b(g,x)b(h,g^{-1}x)b(gh,x)^{-1}
\end{eqnarray}
holds a.e. $x\in X$.

For brevity, we use $b_g(x):=b(g, x)$ below.

Defining $\hat{b}: G\times X^2\to \mathbb{T}$ by $\hat{b}_g(x,y)=b_g(x)b_g(y)^{-1}$, then (\ref{2*}) implies $\hat{b}_g\sigma^2_g(\hat{b}_h)=\hat{b}_{gh}\in L^{\infty}(X^2, \mathbb{T})$, i.e.   $\hat{b}$ is a 1-cocyle for the diagonal action $\sigma^2: G\curvearrowright X^2$. Since it is a $\mathbb{T}$-cocycle superrigid action, we can find a group homomorphism $\lambda: G\to\mathbb{T}$ and a measurable map $t: X^2\to\mathbb{T}$ such that
\begin{eqnarray*}\label{21}
\hat{b}_g=\lambda_gt\sigma^2_g(t^*)
\end{eqnarray*}or, equivalently:
\begin{eqnarray}\label{22}
b_g(x)=b_g(y)\lambda_gt(x, y)t(g^{-1}x, g^{-1}y)^{-1}
\end{eqnarray}holds a.e. $(x, y)\in X^2.$

Now we claim $t$ splits as a simple tensor, i.e.   there exist measurable functions $\phi, \psi: X\to\mathbb{T}$ such that $t(x, y)=\phi(x)\psi(y)$ holds a.e. $(x, y)\in X^2$.

The proof is essentially the same as the proof of \cite[Theorem 3.4]{Furman}(for the ``absolute case"). We outline it for completeness.

From (\ref{22}), we also have
\begin{eqnarray}\label{23}
b_g(z)=b_g(y)\lambda_gt(z, y)t(g^{-1}z, g^{-1}y)^{-1}.
\end{eqnarray}

Substituting (\ref{22}) in (\ref{23}), we obtain the following identity which holds a.e. $(x, y, z)\in X^3$:
$$b_g(z)=[t(z, y)t(x, y)^{-1}]b_g(x)[t(g^{-1}z, g^{-1}y)t(g^{-1}x, g^{-1}y)^{-1}]^{-1}.$$

Setting $\Phi(x, y, z)\overset{\rm{def}}{=}t(z, y)t(x, y)^{-1}$, the above takes the form
$$b_g(z)=\Phi(x, y, z)b_g(x)\Phi(g^{-1}x, g^{-1}y, g^{-1}z)^{-1}$$
or, equivalently:
$$\Phi(g^{-1}x, g^{-1}y, g^{-1}z)=b_g(x)\Phi(x, y, z)b_g(z)^{-1}.$$
Next, observe that the morphism of the diagonal $G$-actions
$$q: (X^3, \nu^3)\longrightarrow (X^2, \nu^2), q(x, y, z)=(x, z)$$
is relative weakly mixing, as described in \cite[Page 25]{Furman}.

Following the notations in \cite{Furman}, we observe that $\mathbb{T}\in \mathcal{U}_{\rm{fin}}\subset\mathcal{G}_{\rm{inv}}$.  Since the ``cocycle" assumption in the statement of \cite[Lemma 3.2]{Furman} is not used in its proof, we can still apply \cite[Lemma 3.2]{Furman} for the measurable maps $\alpha, \beta: G\times X^2\to \mathbb{T}$ given by $\alpha(g, x, z):=b_g(x), \beta(g, x, z):=b_g(z)$ to conclude that $\Phi(x, y, z)=f(x, z)$, for some measurable map $f: X^2\to\mathbb{T}$. Therefore
$$t(z, y)t(x, y)^{-1}=\Phi(x, y, z)=f(x, z)$$
meaning that a.e. $(x, y, z)\in X^3$, we have
$$t(z, y)=f(x, z)t(x, y).$$
Then there exists some $z_0$, such that $$t(x, y)=t(z_0, y)f(x, z_0)^{-1}$$ holds a.e. $(x, y)\in X^2$.

After defining $\psi(y)=t(z_0, y), \phi(x)=f(x, z_0)^{-1}$, we see that $t(x, y)=\phi(x)\psi(y)$ holds a.e. $(x, y)\in X^2$.

The definition of $\hat{b}$ and (\ref{22}) imply $b_g\otimes b_g^*=\lambda_g(\phi\otimes\psi)(\sigma_g(\phi^*)\otimes\sigma_g(\psi^*))\in L^{\infty}(X)\otimes L^{\infty}(X)$ or, equivalently:
\begin{eqnarray}\label{24}
1\otimes 1=\lambda_g(b_g^*\phi\sigma_g(\phi^*)\otimes b_g\psi\sigma_g(\psi^*)).
\end{eqnarray}
Since we can identify $L^{\infty}(X)\otimes L^{\infty}(X)$ with a subspace of $L^{2}(X)\otimes \overline{L^{2}(X)}$, we may think $1\otimes 1,  b_g^*\phi\sigma_g(\phi^*)\otimes b_g\psi\sigma_g(\psi^*)$ as Hilbert-Schmidt operators on $L^2(X)$. Then, take the vector $1\in L^2(X)$, and apply (\ref{24}), we have 
\begin{eqnarray*}
1=\langle 1,1\rangle 1=(1\otimes 1)1&=&\lambda_gb_g^*\phi\sigma_g(\phi^*)\otimes b_g\psi\sigma_g(\psi^*)(1)\\
&=&\lambda_g\langle 1, b_g^*\psi^*\sigma_g(\psi)\rangle b_g^*\phi\sigma_g(\phi^*)\in L^2(X).
\end{eqnarray*}
The above implies $b_g^*\phi\sigma_g(\phi^*)$ is a constant function in $L^2(X)$, say, $b_g^*\phi\sigma_g(\phi^*)=\mu_g$ for some map $\mu: G\to\mathbb{T}$; hence, $b_g=\phi\sigma_g(\phi^*)\mu_g^{-1}$, substituting this in (\ref{2*}), we deduce $c(g,h)=\mu_g^{-1}\mu_h^{-1}\mu_{gh}$, so $[c]=[0]$.
\end{proof}

\begin{remark}
If $G$ has property (T) and the action $G\curvearrowright X$ is a Bernoulli shift, then it is well-known that $H^2(G, \mathbb{T})$ is a subgroup of $H^2(G\curvearrowright X;\mathbb{T})$ by \cite[Theorem 3.1]{Sysak-Popa}.
\end{remark}
\begin{remark}
If the diagonal action $\sigma^2: G\curvearrowright (X^2, \nu^2)$ is weakly mixing and $\mathbb{T}$-cocycle superrigid, then $\sigma: G\curvearrowright (X, \nu)$ is also $\mathbb{T}$-cocycle superrigid by \cite[Theorem 3.1]{Popa_3}.
\end{remark}

\section{Proof of Theorem \ref{maintheorem}}\label{section3}

Now, we are ready to give the proof of Theorem \ref{maintheorem}.

\begin{proof}
We may identify $H^n(G, \mathbb{T})$ with a subgroup of $H^n(G\curvearrowright X;\mathbb{T})$ for $n=1, 2$ by Lemma \ref{lemma2}. Then we can define a group homomorphism
$$\Phi_n: H^n(G, \widehat{X})\to\frac{H^n(G\curvearrowright X;\mathbb{T})}{H^n(G, \mathbb{T})}$$
by $\Phi_n([c])=[\tilde{c}]+H^n(G, \mathbb{T})$, where $\tilde{c}: G^n\times X\to\mathbb{T}$ is defined by $\tilde{c}(g_1,\cdots, g_n, x)=c(g_1,\cdots, g_n)(x)$ for all $g_1, \cdots, g_n\in G$ and $x\in X$.

It is clear that $\Phi_n$ is a well-defined group homomorphism.

\textbf{Case $n=1$}.

We claim $\Phi_1$ is injective if the action $G\curvearrowright X$ is weakly mixing.

Note that $H^1(G, \mathbb{T})=Hom(G, \mathbb{T})$.

If $[c]\in H^1(G,\widehat{X})$ and $\Phi_1([c])=[0]+H^1(G, \mathbb{T})$, then there exist a measurable map $b: X\to \mathbb{T}$ and a group homomorphism $\phi: G\to \mathbb{T}$ such that
\begin{eqnarray}\label{eq1}
c(g)(x)=\tilde{c}(g,x)=b(g^{-1}x)b(x)^{-1}\phi(g)
\end{eqnarray}
holds true  for all $g\in G$ and all $x\in Z\subset X$ a conull set.

Then we prove that we can assume $b$ is a continuous group homomorphism and $\phi(g)=1\in\mathbb{T}$ for all $g\in G$.

Note that the addition operation $\Delta: X\times X\to X, \Delta(x,y):=x+y$, induces a left invariant measure $\Delta_*(\mu\times\mu)$ on $X$ and hence $\Delta_*(\mu\times\mu)=\mu$ by the uniqueness of Haar measure on $X$.
So $W:=\Delta^{-1}(Z)\cap (Z\times Z)\subset X\times X$ is still a conull set and we may assume it is $G$-invariant.

Since $c(g)(x+y)=c(g)(x)c(g)(y)$ for all $g\in G$ and all $x, y\in X$, by (\ref{eq1}) we deduce that
\begin{eqnarray*}\label{eq2}
b(x+y)b(x)^{-1}b(y)^{-1}\phi(g)=(gb)(x+y)[(gb)(x)]^{-1}[(gb)(y)]^{-1}
\end{eqnarray*}
holds true for all $g\in G$ and all $(x,y)\in W\subset X\times X$.

Now we consider the map $\hat{b}: X^2\to \mathbb{T}$ defined by $\hat{b}(x,y)=b(x+y)b(x)^{-1}b(y)^{-1}$, the above equality implies
\begin{eqnarray}\label{eq3}
\hat{b}(x,y)\phi(g)=(g\hat{b})(x,y)
\end{eqnarray}
holds true for all $g\in G$ and all $(x,y)\in W\subset X\times X$.

Notice that the above implies the subspace $\mathbb{C}\hat{b}$ is invariant under the action of $G$. Since $G\curvearrowright X^2$ is weakly mixing, it follows that $\hat{b}$ is a constant function in $L^2(X^2)$.

Hence, $\hat{b}(x,y)\equiv \lambda$ a.e. for some constant $\lambda\in\mathbb{T}$, i.e. $b(x+y)=b(x)b(y)\lambda$ for all $(x,y)\in Z'\subset X\times X$ a conull set.

From (\ref{eq3}), we find $\phi(g)=1\in \mathbb{T}$ for all $g\in G$.

After replacing $b$ by $b\lambda$, which does not change the equality in (\ref{eq1}), we may assume $b$ is an almost group homomorphism in the sense that $b(x+y)=b(x)b(y)$ for all $(x,y)\in Z'\subset X\times X$ a conull set.

Now, Lemma \ref{lemma1} shows that there is a continuous genuine group homomorphism $\tilde{b}: X\to \mathbb{T}$ such that $\tilde{b}(x)=b(x)$  for all $x\in Y\subset X$ a conull set.

So, for any fixed $g\in G$, $c(g)(x)=\tilde{b}(g^{-1}x)\tilde{b}(x)^{-1}$ holds for all $x$ in a conull subset of $X$. Since both sides are continuous functions on $x$, the equality holds everywhere. Hence $c(g)=(g-1)\tilde{b}\in \widehat{X}$, $c$ is a trivial 1-cocycle.

\textbf{Case $n=2$.}

We claim $\Phi_2$ is injective if the diagonal actions $G\curvearrowright X^2, X^4$ are weakly mixing and $\mathbb{T}$-cocycle superrigid.

If $[c]\in H^2(G,\widehat{X})$ and $\Phi_2([c])=[0]+H^2(G, \mathbb{T})$, then there is a measurable map $t:G\times X\to \mathbb{T}$ and a 2-cocycle $\theta: G\times G\to \mathbb{T}$ such that
\begin{eqnarray}\label{eq4}
c(g,h)(x)=\theta(g,h)t(g,x)t(h,g^{-1}x)t(gh,x)^{-1}
\end{eqnarray}
holds true for all $g,h\in G$ and almost every $x\in X$.

For brevity, we use $t_g(x):=t(g, x)$ below.

Now we show that $\theta\in B^2(G,\mathbb{T})$ and $t$ can be replaced by a map from $G$ to $\widehat{X}$.

Since $c(g,h)\in \widehat{X}$, we have the following equality
\begin{eqnarray*}\label{eq5}
t_{gh}(x+y)t_{gh}(x)^{-1}t_{gh}(y)^{-1}\theta(g, h)=
t_g(x+y)t_g(x)^{-1}t_g(y)^{-1}t_h(g^{-1}(x+y))t_h(g^{-1}x)^{-1}t_h(g^{-1}y)^{-1}
\end{eqnarray*}
holds true for all $g,h\in G$ and almost every $(x,y)\in X\times X$.

Defining $\hat{t}: G\times X^2\to\mathbb{T}$ by $\hat{t}_g(x,y)=t_g(x+y)t_g(x)^{-1}t_g(y)^{-1}$, then the above becomes
\begin{eqnarray*}\label{eq6}
\theta(g, h)=\hat{t}_g(x,y)\hat{t}_h(g^{-1}x, g^{-1}y)\hat{t}_{gh}(x,y)^{-1}.
\end{eqnarray*}
Therefore, $\Psi_2([\theta])=[0]$ under the homomorphism $\Psi_2: H^2(G, \mathbb{T})\to H^2(G\curvearrowright X^2; \mathbb{T})$.
Now, since we assume the diagonal action $G\curvearrowright X^4$ is weakly mixing and $\mathbb{T}$-cocycle superrigid, $\Psi_2$ is injective by Lemma \ref{lemma2}, hence $[\theta]=[0]\in H^2(G,\mathbb{T})$.

Now, we may assume $\theta(g,h)=1$ in (\ref{eq4}) and find a new expression
\begin{eqnarray}\label{eq4'}
c(g,h)(x)=t_g(x)t_h(g^{-1}x)t_{gh}(x)^{-1},
\end{eqnarray}
which holds true for all $g,h\in G$ and almost every $x\in X$.

Repeating the above argument, we would conclude
\begin{eqnarray*}\label{eq7}
\hat{t}_g(x,y)\hat{t}_h(g^{-1}x, g^{-1}y)=\hat{t}_{gh}(x,y).
\end{eqnarray*}
Hence, $\hat{t}\in Z^1(G\curvearrowright X^2;\mathbb{T})$. Then we apply the assumption that $G\curvearrowright X^2$ is $\mathbb{T}$-cocycle superrigid to find a measurable map $u: X^2\to \mathbb{T}$ and a group homomorphism $\lambda: G\to\mathbb{T}$ such that
\begin{eqnarray*}
\hat{t}_g=\lambda_gu\sigma^2_g(u^*)
\end{eqnarray*}
holds for all $g\in G$. This is equivalent to say
\begin{eqnarray}\label{eq8}
t_g(x_1+x_2)=t_g(x_1)t_g(x_2)\lambda_gu(x_1, x_2)u(g^{-1}x_1, g^{-1}x_2)^{-1},
\end{eqnarray}
holds for all $(x_1, x_2)\in E\subset X^2$, where $E\subset X^2$ is a conull set, and all $g\in G$.

Note that we may replace $E$ by the smaller conull subset $\cap_{\tau\in S_2}\tau(E)$, where $S_2$ denotes the permutation group of two elements and $(x_1, x_2)\in\tau(E)\overset{\rm{def}}{\Longleftrightarrow}(x_{\tau(1)}, x_{\tau(2)})\in E$, so without loss of generality, we may assume $E$ is symmetric.

Now we claim that we can assume $u$ is a measurable symmetric 2-cocycle, i.e. there exist conull sets $F''\subset X^3$ and $E''\subset X^2$ such that the following hold:
\begin{eqnarray*}
u(x_1, x_2)u(x_1+x_2, x_3)&=&u(x_2, x_3)u(x_1, x_2+x_3),~~~~ \text{for all $(x_1, x_2, x_3)\in F''$};\\
u(x_1, x_2)&=&u(x_2, x_1),~~~~ \text{for all $(x_1, x_2)\in E''$}.
\end{eqnarray*}

First, we show $u$ can be assumed to be symmetric.

Given any $(x_1, x_2)\in E$, we have the following equality by (\ref{eq8})
\begin{eqnarray*}\label{eq8_2}
t_g(x_1+x_2)=t_g(x_1)t_g(x_2)\lambda_gu(x_2, x_1)u(g^{-1}x_2, g^{-1}x_1)^{-1}.
\end{eqnarray*}
Comparing it with (\ref{eq8}), we deduce
\begin{eqnarray*}
u(x_1, x_2)u(x_2, x_1)^{-1}=u(g^{-1}x_1, g^{-1}x_2)u(g^{-1}x_2, g^{-1}x_1)^{-1}.
\end{eqnarray*}
Since $G\curvearrowright X^2$ is weakly mixing, $u(x_1, x_2)=au(x_2, x_1)$ for some constant $a\in\mathbb{T}$ and for all $(x_1, x_2)\in E'\subset X^2$ a conull set.

Defining $E''=\cap_{\tau\in S_2}\tau(E\cap E')$, then $E''$ is still conull.

Note that $E''$ is symmetric; hence $u(x_1, x_2)=au(x_2, x_1)=a^2u(x_1, x_2)$, $a=\pm1$.

If $a=-1$, $u(x_1, x_2)^2=u(x_2, x_1)^2$, and we may first study 2-cocycle $c^2$ (and hence $t_g^2, \lambda_g^2, u^2$ respectively in (\ref{eq8})), and continue the proof below. In the end, we could show $[c^4]=[0]$, but this implies $[c]=[0]$ since $H^2(G, \widehat{X})$ is torsion free, so without loss of generality, we may assume $u(x_1, x_2)=u(x_2, x_1)$ for all $(x_1, x_2)\in E''$, i.e.   $u$ is symmetric.

Second, we show that we can assume $u$ satisfies the 2-cocycle identity.

Defining $F\subset X^3$ by $F=\phi^{-1}(E'')\cap\psi^{-1}(E'')\cap p^{-1}(E'')$, where $\phi, \psi, p: X^3\to X^2$ are defined by $\phi(x_1, x_2, x_3)=(x_1, x_2+x_3), \psi(x_1, x_2, x_3)=(x_2+x_3, x_1)$ and $p(x_1, x_2, x_3)=(x_2, x_3)$ respectively.

Note that $F$ is a measurable conull subset of $X^3$ and we can further assume $F=\tau(F)$ for all $\tau\in S_3$, i.e.   $(x_1, x_2, x_3)\in F$ if and only if $(x_{\tau(1)}, x_{\tau(2)}, x_{\tau(3)})\in F$.

Given any $(x_1, x_2, x_3)\in F$, applying (\ref{eq8}) to $(x_1+x_2, x_3)$, we deduce the following identity
\begin{equation}\label{eq9}
t_g(x_1+x_2+x_3)=t_g(x_1+x_2)t_g(x_3)\lambda_gu(x_1+x_2, x_3)u(g^{-1}(x_1+x_2), g^{-1}x_3)^{-1}.
\end{equation}
Substituting $t_g(x_1+x_2)$ in (\ref{eq9}), via the formula (\ref{eq8}), we further deduce
\begin{equation}\label{eq10}
\frac{t_g(x_1+x_2+x_3)}{t_g(x_1)t_g(x_2)t_g(x_3)\lambda_g^2}=u(x_1, x_2)u(g^{-1}x_1, g^{-1}x_2)^{-1}u(x_1+x_2, x_3)u(g^{-1}(x_1+x_2), g^{-1}x_3)^{-1}.
\end{equation}

Similarly, we could deduce the following
\begin{equation}\label{eq11}
\frac{t_g(x_1+x_2+x_3)}{t_g(x_1)t_g(x_2)t_g(x_3)\lambda_g^2}=u(x_1, x_3)u(g^{-1}x_1, g^{-1}x_3)^{-1}u(x_1+x_3, x_2)u(g^{-1}(x_1+x_3), g^{-1}x_2)^{-1}.
\end{equation}

Now, combining (\ref{eq10}) and (\ref{eq11}), we deduce the following equality,
\begin{eqnarray*}
\frac{u(x_1, x_2)u(x_1+x_2, x_3)}{u(x_1, x_3)u(x_1+x_3, x_2)}=\frac{u(g^{-1}x_1, g^{-1}x_2)u(g^{-1}(x_1+x_2), g^{-1}x_3)}{u(g^{-1}x_1, g^{-1}x_3)u(g^{-1}(x_1+x_3), g^{-1}x_2)}.
\end{eqnarray*}
Since $G\curvearrowright X^3$ is weakly mixing, we deduce that
\begin{eqnarray}\label{eq12}
\frac{u(x_1, x_2)u(x_1+x_2, x_3)}{u(x_1, x_3)u(x_1+x_3, x_2)}=a'
\end{eqnarray}
for some constant $a'\in\mathbb{T}$ and all $(x_1, x_2, x_3)\in F'\subset X^3$, where $F'$ is a conull subset of $X^3$.

Defining $F''=\cap_{\tau\in S_3}\tau(F\cap F')$, then $F''$ is still conull.

Swapping $x_2$ and $x_3$ in (\ref{eq12}), we also have $\frac{u(x_1, x_3)u(x_1+x_3, x_2)}{u(x_1, x_2)u(x_1+x_2, x_3)}=a'$, hence $a'^2=1$, $a'=\pm1$. If $a'=-1$, we study $u^2$ and $c^2$ first as before, so without loss of generality, we may assume $a'=1$;

Now, we have shown that $u(x_1, x_2)u(x_1+x_2, x_3)=u(x_1, x_3)u(x_1+x_3, x_2)$ for all $(x_1, x_2, x_3)\in F''$.

Since $u$ is symmetric and $X$ is abelian, we can also rewrite it
as $$u(x_2, x_1)u(x_2+x_1, x_3)=u(x_1, x_3)u(x_2, x_1+x_3)$$ or, equivalently: $$u(x_1, x_2)u(x_1+x_2, x_3)=u(x_2, x_3)u(x_1, x_2+x_3)$$ for all $(x_1, x_2, x_3)\in F''$.

From above, we have shown that $u$ can be assumed to be a measurable symmetric 2-cocycle. Since $X$ is a compact separable abelian group, any such a 2-cocycle is a (measurable) 2-coboundary (see e.g. \cite[Theorem 10]{Moore_1} and \cite[page 39]{Moore_2} or see \cite[line -4 on page 252]{Neshveyev_Stormer}), i.e.   there exists a measurable map $v: X\to\mathbb{T}$ such that
$u(x_1, x_2)=v(x_1)v(x_2)v(x_1+x_2)^{-1}$ for a.e. $(x_1, x_2)\in X^2$.

Notice that (\ref{eq8}) can be written as
\begin{equation*}
t_g(x_1+x_2)=t_g(x_1)t_g(x_2)\lambda_gv(x_1)v(x_2)v(x_1+x_2)^{-1}[v(g^{-1}x_1)v(g^{-1}x_2)v(g^{-1}x_1+g^{-1}x_2)^{-1}]^{-1}
\end{equation*} or, equivalently:
\begin{eqnarray*}
[t_g(x_1)v(x_1)v(g^{-1}x_1)^{-1}][t_g(x_2)v(x_2)v(g^{-1}x_2)^{-1}]=\lambda_g^{-1}[t_g(x_1+x_2)v(x_1+x_2)v(g^{-1}x_1+g^{-1}x_2)^{-1}].
\end{eqnarray*}
Then after replacing $t_g$ by $t_gv\sigma_g(v^*)\lambda_g$, which does not change the equality in (\ref{eq4'}), and applying Lemma \ref{lemma1}, we may assume $t_g\in \widehat{X}$ for all $g\in G$; hence (\ref{eq4'}) holds for all $x\in X, g, h\in G$ with $t_g\in \widehat{X}$, $[c]=[0]$.

\end{proof}

\begin{remark}
The above proof (for $n=2$) starting from (\ref{eq8}) is similar to step 5 in the proof of \cite[Theorem 8.2]{Ioana_Popa_Vaes}.
\end{remark}

\section{Applications}\label{section4}

Now, we give some applications of our theorem.

We denote by $[k]$ a finite set with $k$ points, and denote by $\nu_k$ the uniformly distributed measure on $[k]$.

\begin{corollary}\label{cor1}
If $G\in \mathcal{G}$, e.g. $G$ has Property (T), then $H^2(G, \mathbb{Z}G)$ is torsion free as an abelian group.
\end{corollary}
\begin{proof}
For any $k\geq 2$, we take $f=k\in \mathbb{Z}G$, then note that $G\curvearrowright X_f=\widehat{\frac{\mathbb{Z}G}{\mathbb{Z}Gf}}\cong [k]^G$ is a Bernoulli shift, which is mixing. Notice that $G\curvearrowright ([k]^G,\nu_k^G)$ is $\mathbb{T}$-valued cocycle superrigid implies $H^1(G\curvearrowright [k]^G; \mathbb{T})=H^1(G, \mathbb{T})$, then applying Theorem \ref{maintheorem}, we deduce
$H^1(G, \frac{\mathbb{Z}G}{\mathbb{Z}Gf})=0$.

Note that we have a short exact sequence $0\to\mathbb{Z}Gk\to\mathbb{Z}G\to\frac{\mathbb{Z}G}{\mathbb{Z}Gk}\to 0$ of $G$-modules and $\mathbb{Z}Gf\cong\mathbb{Z}G$ whenever $f\in\mathbb{Z}G$ is not a right zero-divisor. By the well-known fact in homological algebra \cite[Chapter III, Proposition 6.1]{GTM_87}, we know the natural map $H^2(G, \mathbb{Z}G)\overset{\times k}{\rightarrow} H^2(G, \mathbb{Z}G)$ is injective for any positive integer $k$, hence $H^2(G, \mathbb{Z}G)$ is torsion free.
\end{proof}
\begin{remark}
We do not know a direct proof of this corollary when $G$ has Property (T) without using Popa's cocycle superrigidity theorem \cite{Popa_3}. $H^2(G, \mathbb{Z}G)$ is torsion free for any finitely presented group $G$ \cite[Proposition 13.7.1]{GTM_243}, but a group $G$ with (T) might not be finitely presented, e.g.  $G=SL_3(\mathbb{F}_p[X])$ \cite[Section 3.4]{Kazhdan_T}.
\end{remark}

\begin{corollary}\label{cor2}
Considering the Bernoulli shift action $G\curvearrowright (\mathbb{T}^G, \nu^G)$ or $G\curvearrowright ([k]^G, {\nu_k}^G)$ with $G\in\mathcal{G}$, where $\nu$ is the Haar measure on $\mathbb{T}$, the following are true.
\begin{enumerate}
  \item $H^n(G\curvearrowright \mathbb{T}^G;\mathbb{T})=H^n(G, \mathbb{T})$ implies $H^n(G, \mathbb{Z}G)=0$ for $n=1, 2$.
  \item $H^1(G\curvearrowright [k]^G;\mathbb{T})=H^1(G, \mathbb{T})$ implies $H^1(G, \frac{\mathbb{Z}G}{\mathbb{Z}Gk})=0$.
  \item $H^2(G, \frac{\mathbb{Z}G}{\mathbb{Z}Gk})=0$ implies $H^2(G,\mathbb{Z}G)$ is $k$-divisible, i.e.   the multiplication by $k$ map on $H^2(G,\mathbb{Z}G)$ is surjective.
  \item $H^1(G, \frac{\mathbb{Z}G}{\mathbb{Z}Gk})=0$ implies $H^{2}(G,\mathbb{Z}G)$ has no (non-trivial) elements with order $k$.
\end{enumerate}
\end{corollary}
\begin{proof}
Let $\widehat{X}=\mathbb{Z}G$, by Corollary \ref{cor1}, we can apply Theorem \ref{maintheorem} for $n=1, 2$ to see (1) holds.

Let $\widehat{X}=\frac{\mathbb{Z}G}{\mathbb{Z}Gk}$, we apply Theorem \ref{maintheorem} for $n=1$ to see (2) holds.

For (3) and (4), the proof is similar to the proof of Corollary \ref{cor1}. In fact, (3) and (4) hold without the assumption $G\in\mathcal{G}$.
\end{proof}

For a Bernoulli shift action $G\curvearrowright X$ of a property (T) group $G$, Sorin Popa asked whether $H^n(G\curvearrowright X;\mathbb{T})=H^n(G, \mathbb{T})$ holds for $n>1$ \cite[Section 6.6]{Popa_3}; unfortunately, this is not true for $n=2$ in full generality.
\begin{corollary}
There exists a countable discrete group $G$ with property (T) such that $H^2(G\curvearrowright \mathbb{T}^G;\mathbb{T})\neq H^2(G, \mathbb{T})$.
\end{corollary}
\begin{proof}
By Corollary \ref{cor2} (1) for $n=2$, we only need to find a group $G$ with (T) and non-trivial $H^2(G,\mathbb{Z}G)$. Then by \cite[Chapter VIII, Proposition 6.7]{GTM_87}, it suffices to find a group $G$ with (T) such that it is of type $FP$ and has cohomological dimension $cd(G)=2$. Such a group exists, see e.g. \cite[Theorem 1]{Olliver}\cite[Section 9.B]{Gromov}\cite[Section I.2.b., Theorem 11; Section I.3.g., Theorem 27]{Random_group}. Indeed, we can take $G$ to be a random group in Gromov's density model at density $d\in (\frac{1}{3}, \frac{1}{2})$. It is with overwhelming probability torsion free, word-hyperbolic, of cohomological dimension 2 \cite[Section I.3.b]{Random_group}, (hence of type $FP$ by \cite[Proposition 6.3(3)]{Kapovich} and \cite[Chapter VIII, Proposition 6.1]{GTM_87}) and has property (T) by Zuk \cite{Zuk} (further clarified by Kotowski-Kotowski \cite{Kotowski}).
\end{proof}

\begin{remark}
The important problem of computing $H^2(G\curvearrowright X; \mathbb{T})$ explicitly for a Bernoulli action of various examples of property (T) groups remains open. In particular, it seems unclear whether $H^2(G\curvearrowright X; \mathbb{T})=H^2(G, \mathbb{T})$ holds for $G=SL_3(\mathbb{Z})$.
\end{remark}

\section{Further Questions}\label{section5}
We end this paper by discussing a few open questions which arose from our investigation.

(1) Can we extend \cite[Corollary 4.2]{Peterson-Sinclair} to second cohomology groups?

(2) Can we find a group $G\in \mathcal{G}$  and $f\in \mathbb{Z}G$ such that the principal algebraic action $G\curvearrowright X_f:=\widehat{\frac{\mathbb{Z}G}{\mathbb{Z}Gf}}$ is weakly mixing and has nontrivial $H^1(G,\frac{\mathbb{Z}G}{\mathbb{Z}Gf})$, or in general, $H^1(G\curvearrowright X_f; \mathbb{T})\neq H^1(G, \mathbb{T})$? Can we find such an example with $f\in \mathbb{Z}G\cap\ell^1(G)^{\times}$? If such an example exists, it would not be conjugate to a Bernoulli shift. Note that
$H^1(G,\frac{\mathbb{Z}G}{\mathbb{Z}Gf})=0$ is equivalent to the right multiplication by $f$ map $R_f: H^2(G, \mathbb{Z}G)\to H^2(G, \mathbb{Z}G)$ is injective for $G\in\mathcal{G}$ and $f\in\mathbb{Z}G$ is not a right-zero divisor. Also note that the right $G$-module structure on $H^2(G,\mathbb{Z}G)$ might not be trivial, which was pointed out to us by Boris Okun.

(3) Does there exist a weakly mixing principal algebraic action $G\curvearrowright X_f:=\widehat{\frac{\mathbb{Z}G}{\mathbb{Z}Gf}}$ such that $0\neq H^2(G, \frac{\mathbb{Z}G}{\mathbb{Z}Gf})=\frac{H^2(G\curvearrowright X_f; \mathbb{T})}{H^2(G, \mathbb{T})}$? This was suggested to us by the referee.

\bibliographystyle{amsplain}

\end{document}